\documentclass[12pt,a4paper,reqno]{amsart} 
\usepackage{amssymb}
\usepackage{latexsym}
\usepackage{amsmath}
\usepackage{mathrsfs}
\usepackage[X2,T1]{fontenc}
\usepackage{cite}
\usepackage{calc}                   

\newcommand{\scal}[2]{\langle #1,#2\rangle}
\newcommand{\rr}[1]{\mathbf R^{#1}}

\newcommand{\mabfp}{{\boldsymbol p}}
\newcommand{\mabfq}{{\boldsymbol q}}

\newcommand{\cc}[1]{\mathbf C^{#1}}
\newcommand{\zz}[1]{\mathbf Z^{#1}}
\newcommand{\nm}[2]{\Vert #1\Vert _{#2}}

\newcommand{\sets}[2]{\{ \, #1\, ;\, #2\, \} }

\newcommand{\ep}{\varepsilon}
\newcommand{\fy}{\varphi}
\newcommand{\cdo}{\, \cdot \, }

\newcommand{\essup}{\operatorname{ess\, sup}}

\newcommand{\vrum}{\vspace{0.1cm}}

\newcommand{\rd}{\mathbf{R} ^{d}}

\newcommand{\im}{i}

\newcommand{\nn}[1]{{\mathbf N}^{#1}}

\newcommand{\maclH}{\mathcal H}

\newcommand{\maclS}{\mathcal S}
\newcommand{\mascB}{\mathscr B}

\newcommand{\mascF}{\mathscr F}

\newcommand{\mascP}{\mathscr P}

\newcommand{\mascS}{\mathscr S}

\numberwithin{equation}{section}          

\newtheorem{thm}{Theorem}
\numberwithin{thm}{section}

\newcommand{\rubrik}{}
\newtheorem{prop}[thm]{Proposition}
\newtheorem{cor}[thm]{Corollary}
\newtheorem{lemma}[thm]{Lemma}

\theoremstyle{definition}

\newtheorem{defn}[thm]{Definition}
\newtheorem{example}[thm]{Example}

\theoremstyle{remark}
\newtheorem{rem}[thm]{Remark}

\author{Christine Pfeuffer}

\address{Department of Mathematics,
University of Regensburg, Regensburg, Germany}

\email{christine.pfeuffer@mathematik.uni-regensburg.de}

\author{Joachim Toft}

\address{Department of Mathematics,
Linn{\ae}us University, V{\"a}xj{\"o}, Sweden}

\email{joachim.toft@lnu.se}

\title{Compactness properties for modulation spaces}


\subjclass[2010]{}

\begin{document}

\begin{abstract} 
We prove that if $\omega _1$ and $\omega _2$ are moderate
weights and $\mascB$ is a suitable (quasi-)Banach function space, then
a necessary and sufficient condition for the embedding 
\ $i\, :\, M (\omega _1,\mascB )\to M (\omega _2,\mascB )$  
between two modulation spaces to be compact is that the quotient \
$\omega _2/\omega _1$ vanishes at infinity. Moreover we show, that the boundedness of $\omega _2/\omega _1$ a necessary and sufficient condition for the previous embedding to be continuous. 
\end{abstract}

\maketitle

\section{Introduction}\label{sec0}

\par

In the paper we extend well-known compact embedding properties for classical
modulation spaces to a broader family of modulation spaces. These investigations
go in some sense back to \cite{Shubin}, where M. Shubin proved that if $t>0$, then
the embedding $i\! : \! Q_s\to Q_{s-t}$ is compact. In the community,
the previous compactness property was not obvious since any similar fact does not
hold when the Shubin spaces $Q_s$ and $Q_{s-t}$ are replaced by the Sobolev spaces
$H^2_s$ and $H^2_{s-t}$ of Hilbert types. Since
\begin{alignat}{2}
Q_s &= M^{2,2}_{(\omega )},& \qquad \omega (X) &= (1+|x|+|\xi |)^s
\label{Eq:ShubinMod}
\intertext{and}
H^2_s &= M^{2,2}_{(\omega )},& \qquad \omega (X) &= (1+|\xi |)^s,
\qquad  X=(x,\xi )
\label{Eq:SobolevMod}
\end{alignat}
the previous compact embedding properties can also be
written by means of modulation spaces. In this context, a more general
situation were considered by M. D{\"o}rfler,
H. Feichtinger and K. Gr{\"o}chenig who proved in \cite[Theorem 5]{DorFeiGro}
that if $p,q\in [1,\infty )$, and $\omega _1$ and
$\omega _2$ are certain moderate weights of polynomial types, then
\begin{equation}\label{eq0.5}
i\! : \! M^{p,q}_{(\omega _1)}(\rr d)\to M^{p,q}_{(\omega _2)}(\rr d)
\end{equation}
is compact if and only if $\omega _2/\omega _1$ tends to
zero at infinity. By choosing $\omega _j$ in similar ways as in
\eqref{Eq:ShubinMod}, the latter compactness result confirms the compactness
of embedding $i\! : \! Q_s\to Q_{s-t}$ above by Shubin, as well as confirms
the lack of compactness of the embedding $i\! : \! H^2_s\to H^2_{s-t}$ for Sobolev
spaces.
%

\par

In \cite{BoTo}, the compact embedding property
\cite[Theorem 5]{DorFeiGro} by D{\"o}rfler,
Feichtinger and Gr{\"o}chenig were extended in such
ways that all moderate weights $\omega _j$ of polynomial type are
included. That is, there are no other restrictions on $\omega _j$ than there
should exists constants $N_j>0$ such that
\begin{equation}\label{Eq:PolMod}
\omega _j(X+Y)\lesssim \omega _j(X)(1+|Y|)^{N_j},\qquad j=1,2.
\end{equation}
Moreover, in \cite{BoTo}, the Lebesgue exponents $p$ and $q$ are allowed to
attain $\infty$.

\par

In Section \ref{sec2} we extend these results to involve modulation
spaces $M(\omega ,\mascB)$, which are more general in different ways.
Firstly, there are no boundedness estimates  of polynomial type for the involved
weight $\omega$. In most of our considerations, we require that the weights are
moderate, which impose boundedness estimate of exponential types
$$
\omega _j(X+Y)\lesssim \omega _j(X)e^{r_j(|Y|)},\qquad j=1,2,
$$
for some constants $c_1,c_2>0$. We notice that the latter estimate is less restrective
than the condition \eqref{Eq:PolMod}, which is assumed in \cite{BoTo,DorFeiGro}.

\par

Secondly, 
$\mascB$ can be any general translation invariant Banach function space without
restrictions that $M(\omega ,\mascB)$ should be of the form $M^{p,q}_{(\omega )}$.
We may also have $M(\omega ,\mascB)=M^{p,q}_{(\omega )}$, but in contrast to. 
\cite{BoTo,DorFeiGro}, we here allow
$p$ and $q$ to be smaller than $1$. Here we
notice that if $p<1$ or $q<1$, then $M^{p,q}_{(\omega )}$ fails to be a Banach space
because of absence of convex topological structures.

\par

Thirdly, we show that \eqref{eq0.5} is compact when $\omega _2/\omega _1$ tends
to zero at infinity, and the conditions on $\omega _1$ and $\omega _2$ are relaxed
into a suitable "local moderate condition" (cf. Theorem \ref{compact} (1)). We refer to
\cite{To18} and to some extent to \cite{To22} for a detailed study of modulation
spaces with such relaxed conditions assumptions on the involved weight functions.

\par

Finally we remark that compactness properties for \eqref{eq0.5} can also be obtained
by Gabor analysis, which transfers \eqref{eq0.5} into
$$
i\! : \! \ell ^{p,q}_{(\omega _1)}\to \ell ^{p,q}_{(\omega _2)},
$$
provided $\omega _1$ and $\omega _2$ are moderate weights. Since it is clear that
the latter inclusion map is compact, if and only if $\omega _2/\omega _1$ tends to
$0$ at infinity. Hence the compactness results in \cite{Shubin,BoTo} as well as some of
the results in Section \ref{sec2} can be deduced in such ways. We emphasise however
that such technique can not be used in those situations in Section \ref{sec2}
when modulation spaces are of the form
$M(\omega ,\mascB )$, where either $\mascB$ is a general BF-space, or $\omega$ fails
to be moderate, since the Gabor analysis seems to be insufficient in such situations.

\par

\section{Preliminaries}\label{sec1}

In this section we discuss basic properties for modulation
spaces and other related spaces. The proofs are in many cases omitted
since they can be found in \cite
{Fe2,Fe3,Fe4,FG1,FG2,FG4,Gc2,To1,To2,To7,To8}.

\par

\subsection{Weight functions}\label{subsec1.1}

A \emph{weight} or \emph{weight function} $\omega$ on $\rr d$ is a
positive function such that $\omega ,1/\omega \in
L^\infty _{loc}(\rr d)$. Let $\omega$ and $v$ be weights on $\rr d$.
Then $\omega$ is called \emph{$v$-moderate} or \emph{moderate},
if
\begin{equation}\label{e1.1}
\omega (x_1+x_2)\lesssim \omega (x_1) v(x_2),\quad x_1,x_2\in \rr d .
\end{equation}
Here $f(\theta )\lesssim g(\theta )$ means that $f(\theta )\le cg(\theta)$ for some
constant $c>0$ which is independent of $\theta$ in the domain of $f$ and $g$.
If $v$ can be chosen as polynomial, then $\omega$ is called a weight of
polynomial type.

\par

The function $v$ is called \emph{submultiplicative}, if
it is even and \eqref{e1.1} holds for $\omega =v$. We notice that \eqref{e1.1}
implies that if $v$ is submultiplicative on $\rr d$, then there is a constant $c>0$ such that
$v(x)\ge c$ when $x\in \rr d$.

\par

We let $\mascP _E(\rr d)$ be the set of all moderate weights on
$\rr d$, and $\mascP (\rr d)$ be the subset of $\mascP  _E(\rr d)$
which consists of all polynomially moderate functions on $\rr d$.
We also let $\mascP _{E,s}(\rr d)$ ($\mascP _{E,s}^0(\rr d)$) be the set of
all weights $\omega$ in $\rr d$ such that
\begin{equation}\label{Eq:ModWeightProp}
\omega (x_1+x_2)\lesssim \omega (x_1) e^{r|x_2|^{\frac 1s}},\quad x_1,x_2\in \rr d .
\end{equation}
for some $r>0$ (for every $r>0$). We have
\begin{alignat*}{3}
\mascP &\subseteq \mascP _{E,s_1}^0\subseteq \mascP _{E,s_1}\subseteq
\mascP _{E,s_2}^0\subseteq \mascP _E & \quad &\text{when} & s_2&<s_1 
\intertext{and}
\mascP _{E,s} &= \mascP _E & \quad &\text{when} &\quad  s&\le 1 ,
\end{alignat*}
where the last equality follows from the fact that if $\omega \in \mascP _E(\rr d)$
($\omega \in \mascP _E^0(\rr d)$), then
\begin{equation}\label{Eq:ModWeightPropCons}
\omega (x+y)\lesssim \omega (x) e^{r|y|^{\frac 1s}}
\quad \text{and}\quad
e^{-r|x|}\le \omega (x)\lesssim e^{r|x|},\quad
x,y\in \rr d
\end{equation}
hold true for some $r>0$ (for every $r>0$) (cf. \cite{Gc2.5}).

%

\par

In some situations we shall consider a more general class of weights compared
to $\mascP _E$. (Cf. \cite[Definition 1.1]{To18}.)

\par

\begin{defn}\label{Def:PQweights}
The set $\mascP  _{Q}(\rr d)$ consists of all weights
$\omega$ on $\rr d$ such that
\begin{multline}\label{Eq:ModRelax}
\omega (x)^2\lesssim \omega (x+y)\omega (x-y)\lesssim \omega (x)^2
\\[1ex]
\text{when}\quad
Rc \le |x|\le \frac c{|y|} ,\quad R\ge 2,
\end{multline}
\begin{equation}\label{Gaussest}
e^{-r|x|^2}\lesssim \omega (x)\lesssim e^{r|x|^2},
\end{equation}
holds for some positive constants $c$ and $r$.
\end{defn}

\par


\par

\subsection{Gelfand-Shilov spaces}\label{subsec1.2}

\par

First of all let us fix $0<h,s,t\in \mathbf R$ for the whole subsection. Then we denote the set of 
 all functions $f\in C^\infty (\rr d)$ such that
\begin{equation}\label{gfseminorm}
\nm f{\mathcal S_{t,h}^s}\equiv \sup \frac {|x^\beta \partial ^\alpha
f(x)|}{h^{|\alpha  + \beta |}\alpha !^s\, \beta !^t} < \infty
\end{equation}
by  $\mathcal S_{s,h}(\rr d)$. Here the supremum is taken over all $\alpha ,\beta \in
\mathbf N^d$ and $x\in \rr d$.

\par

One immediately gets, that $\mathcal S_{s,h}^t$ is a Banach space which is contained in $\mascS$.
Moreover $\mathcal S_{s,h}^t$ increases with $h$, $s$ and $t$ and we have the inclusion 
$\mathcal S_{s,h}^t\hookrightarrow \mathscr S$. We use the notation $A\hookrightarrow B$ for topological
spaces $A$ and $B$ satisfying $A\subseteq B$ with continuous embeddings.
Furthermore for sufficiently large $s,t>\frac 12$, or $s =t=\frac 12$ and $h$
$\maclS _{t,h}^s$ contains all finite linear
combinations of the Hermite functions. On account of the density of such linear combinations
in $\mathscr S$ and in $\maclS _{t,h}^s$, the dual $(\mathcal S_{t,h}^s)'(\rr d)$ of $\mathcal S_{t,h}^s(\rr d)$ is
a Banach space which contains $\mathscr S'(\rr d)$, for such choices
of $s$ and $t$.

\par

The inductive and projective limits respectively of $\mathcal S_{t,h}^s(\rr d)$
are called  \emph{Gelfand-Shilov spaces} of Beurling respectively Roumieu
type and are denoted by $\mathcal S_{t}^s(\rr d)$ and
$\Sigma _{t}^s(\rr d)$. Hence
\begin{equation}\label{GSspacecond1}
\mathcal S_t^{s}(\rr d) = \bigcup _{h>0}\mathcal S_{t,h}^s(\rr d)
\quad \text{and}\quad \Sigma _t^{s}(\rr d) =\bigcap _{h>0}
\mathcal S_{t,h}^s(\rr d),
\end{equation}
where the topology for $\mathcal S_t^{s}(\rr d)$ is the strongest
possible one such that the inclusion map from $\mathcal S_{t,h}^s
(\rr d)$ to $\mathcal S_t^{s}(\rr d)$ is continuous, for every choice 
of $h>0$. Equipped with the seminorms $\nm \cdo{\mathcal S_{t,h}^s}$, $h>0$
the space $\Sigma _t^s(\rr d)$ is a Fr{\'e}chet space. Additionally
$\Sigma _t^s(\rr d)\neq \{ 0\}$, if and only if $s+t\ge 1$ and
$(s,t)\neq (\frac 12,\frac 12)$, and $\maclS _t^s(\rr d)\neq \{ 0\}$, if and only
if $s+t\ge 1$.

\medspace

The \emph{Gelfand-Shilov distribution spaces} $(\mathcal S_t^{s})'(\rr d)$
and $(\Sigma _t^s)'(\rr d)$ are the projective and inductive limit
respectively of $(\mathcal S_{t,h}^s)'(\rr d)$.  This implies that
\begin{equation}\tag*{(\ref{GSspacecond1})$'$}
(\mathcal S_t^s)'(\rr d) = \bigcap _{h>0}(\mathcal S_{t,h}^s)'(\rr d)\quad
\text{and}\quad (\Sigma _t^s)'(\rr d) =\bigcup _{h>0}(\mathcal S_{t,h}^s)'(\rr d).
\end{equation}
Note,  that $(\mathcal S_t^s)'(\rr d)$
is the dual of $\mathcal S_t^s(\rr d)$, and $(\Sigma _t^s)'(\rr d)$
is the dual of $\Sigma _t^s(\rr d)$ as proved in \cite{GS}. This is also true
in topological sense. In case $s=t$ we
set
$$
\maclS _s=\maclS _s^s,\quad \maclS _s'=(\maclS _s^s)',\quad
\Sigma _s=\Sigma _s^s
\quad \text{and}\quad
\Sigma _s'=(\Sigma _s^s)'.
$$

\par

For every admissible $s,t>0$ and $\ep >0$ the next embeddings are true:
\begin{equation}\label{GSembeddings}
\begin{alignedat}{2}
\Sigma _t^s (\rr d) &\hookrightarrow &
\maclS _t^s(\rr d) &\hookrightarrow  \Sigma _{t+\ep}^{s+\ep}(\rr d)
\\[1ex]
\quad \text{and}\quad
(\Sigma _{t+\ep}^{s+\ep})' (\rr d) &\hookrightarrow & (\maclS _t^s)'(\rr d)
&\hookrightarrow  (\Sigma _t^s)'(\rr d).
\end{alignedat}
\end{equation}

\par

We recall that Fourier transform of $f\in L^1(\rr d)$ is defined by
$$
(\mathscr Ff)(\xi )= \widehat f(\xi ) \equiv (2\pi )^{-\frac d2}\int _{\rr
{d}} f(x)e^{-i\scal  x\xi }\, dx,
$$
where $\scal \cdo \cdo$ is the usual
scalar product on $\rr d$. The map $\mathscr F$ extends 
uniquely to homeomorphisms on $\mathscr S'(\rr d)$,
from $(\mathcal S_t^s)'(\rr d)$ to $(\mathcal S_s^t)'(\rr d)$ and
from $(\Sigma _t^s)'(\rr d)$ to $(\Sigma _s^t)'(\rr d)$. Furthermore,
$\mascF$ restricts to
homeomorphisms on $\mathscr S(\rr d)$, from
$\mathcal S_t^s(\rr d)$ to $\mathcal S_s^t(\rr d)$ and
from $\Sigma _t^s(\rr d)$ to $\Sigma _s^t(\rr d)$,
and to a unitary operator on $L^2(\rr d)$. If we replace the Fourier transform by a partial
Fourier transform similar results hold true
for $s=t$.

\par

Gelfand-Shilov spaces and their distribution spaces can be characterized in a convenient
way by means of estimates of the short-time Fourier
transforms, see e.{\,}g. \cite{GZ,To18,To22}. Before stating this result, we recall the definition of
the short-time Fourier transform.

\par

For a fixed $\phi \in \maclS _s '(\rr d)$ the \emph{short-time
Fourier transform} $V_\phi f$ of $f\in \maclS _s '
(\rr d)$ with respect to the \emph{window function} $\phi$ is
the Gelfand-Shilov distribution on $\rr {2d}$, defined by
$$
V_\phi f(x,\xi ) \equiv  (\mascF _2 (U(f\otimes \phi )))(x,\xi ) =
\mascF (f \, \overline {\phi (\cdo -x)})(\xi
),
$$
where $(UF)(x,y)=F(y,y-x)$. Here $\mascF _2F$ denotes the partial Fourier transform
of $F(x,y)\in \maclS _s'(\rr {2d})$ with respect to the $y$ variable.
In case $f ,\phi \in \maclS _s (\rr d)$ the short-time Fourier transform of $f$ can be written as  
$$
V_\phi f(x,\xi ) = (2\pi )^{-\frac d2}\int f(y)\overline {\phi
(y-x)}e^{-i\scal y\xi}\, dy .
$$

\par

The characterisation of Gelfand-Shilov functions
and their distributions are formulated in the next two  propositions.
The proof of the characterizations can be found in
e.{\,}g. \cite{GZ,To22} (cf. \cite[Theorem 2.7]{GZ}) and in \cite{To18,To22}:

\par

\begin{prop}\label{stftGelfand2}
Let $s,t,s_0,t_0>0$ be such that $s_0+t_0\ge 1$, 
$s_0\le s$ and $t_0\le t$. Also let
$\phi \in \mathcal S_{t_0}^{s_0}(\rr d)\setminus 0$ and
$f\in (\mathcal S_{t_0}^{s_0})'(\rr d)$.
Then the following is true:
\begin{enumerate}
\item $f\in  \mathcal S_{t}^s(\rr d)$, if and only if
\begin{equation}\label{stftexpest2}
|V_\phi f(x,\xi )| \lesssim  e^{-r (|x|^{\frac 1t}+|\xi |^{\frac 1s})},
\end{equation}
holds for some $r > 0$;

\vrum

\item if in addition $(s_0,t_0)\neq (\frac 12,\frac 12)$ and $\phi \in
\Sigma _{t_0}^{s_0}(\rr d)$, then
$f\in  \Sigma _{t}^{s}(\rr d)$, if and only if \eqref{stftexpest2}
holds for every $r > 0$.
\end{enumerate}
\end{prop}

\par

\begin{prop}\label{stftGelfand2dist}
Let $s,t,s_0,t_0>0$ be such that $s_0+t_0\ge 1$, 
$s_0\le s$ and $t_0\le t$. Also let
$\phi \in \mathcal S_{t}^{s}(\rr d)\setminus 0$ and
$f\in (\mathcal S_{t_0}^{s_0})'(\rr d)$.
Then the following is true:
\begin{enumerate}
\item $f\in  (\mathcal S_{t}^s)'(\rr d)$, if and only if
\begin{equation}\label{stftexpest2Dist}
|V_\phi f(x,\xi )| \lesssim  e^{r(|x|^{\frac 1t}+|\xi |^{\frac 1s})},
\end{equation}
holds for every $r > 0$;

\vrum

\item if in addition $(s_0,t_0)\neq (\frac 12,\frac 12)$ and $\phi \in
\Sigma _{t_0}^{s_0}(\rr d)$, then
$f\in  (\Sigma _{t}^{s})'(\rr d)$, if and only if \eqref{stftexpest2Dist}
holds for some $r > 0$.
\end{enumerate}
\end{prop}

\par

\begin{rem}\label{Rem:STFT}
For the short-time Fourier transform the following continuity results hold: For every $s>0$,
the mapping $(f,\phi)\mapsto V_\phi f$ is continuous from
$\maclS _s(\rr d)\times \maclS _s(\rr d)$ to $\maclS _s(\rr {2d})$
and extends uniquely to continuous mappings from
$\maclS _s'(\rr d)\times \maclS _s'(\rr d)$ to $\maclS _s'(\rr {2d})$.
The same is true if we replace each $\maclS _s$ by
$\mascS$ or by $\Sigma _s$ (cf. e.{\,}g. \cite{Te2,To22}).
\end{rem}

\par







\subsection{Modulation spaces}\label{subsec1.3}

\par

In the whole subsection we fix some $\phi \in \Sigma _1(\rr d)\setminus 0$, $p,q\in (0,\infty ]$
and $\omega \in\mascP _E(\rr {2d})$. Then the
\emph{modulation space} $M^{p,q}_{(\omega )}(\rr d)$ is defined as the set of all
$f\in \Sigma _1'(\rr d)$ such that 
\begin{equation}\label{modnorm}
\nm f{M^{p,q}_{(\omega )}}\equiv \Big (\int \Big (\int |V_\phi f(x,\xi
)\omega (x,\xi )|^p\, dx\Big )^{q/p}\, d\xi \Big )^{1/q} <\infty
\end{equation}
holds.
We set $M^p_{(\omega )}=M^{p,p}_{(\omega )}$. Moreover we use the notion
 $M^{p,q}=M^{p,q}_{(\omega )}$
and $M^{p}=M^{p}_{(\omega )}$ if $\omega =1$.

\par

We summarize some well-known facts of Modulation spaces in the next Proposition. See 
\cite {Fe4,GaSa,Gc2,To20} for the proof. The conjugate exponent of $p$
is given by
$$
p'
=
\begin{cases}
\infty & \text{when}\ p\in (0,1],
\\[1ex]
\displaystyle{\frac p{p-1}} & \text{when}\ p\in (1,\infty ),
\\[2ex]
1 & \text{when}\ p=\infty \, .
\end{cases}
$$

\par

\begin{prop}\label{p1.4}
Let $p,q,p_j,q_j,r\in (0,\infty ]$ be such that $r\le \min (1,p,q)$,
$j=1,2$, let $\omega
,\omega _1,\omega _2,v\in\mascP _E(\rr {2d})$ be such that $\omega$
is $v$-moderate, $\phi \in M^r_{(v)}(\rr d)\setminus 0$, and let $f\in
\Sigma _1'(\rr d)$. Then the following is true:
\begin{enumerate}
\item $f\in
M^{p,q}_{(\omega )}(\rr d)$ if and only if \eqref {modnorm} holds,
i.{\,}e. $M^{p,q}_{(\omega )}(\rr d)$ is independent of the choice of
$\phi$. Moreover, $M^{p,q}_{(\omega )}$ is a quasi-Banach space under the
quasi-norm in \eqref{modnorm} and even a Banach space if $p,q \geq 1$.
Different choices of $\phi$ give rise to
equivalent (quasi-)norms;

\vrum

\item if  $p_1\le p_2$,
$q_1\le q_2$ and $\omega _2\le C\omega _1$ for some constant $C$, then
\begin{alignat*}{3}
\Sigma _1(\rr d)&\subseteq &M^{p_1,q_1}_{(\omega _1)}(\rr
d) &\subseteq  & M^{p_2,q_2}_{(\omega _2)}(\rr d)&\subseteq 
\Sigma _1'(\rr d).
\end{alignat*}
\end{enumerate}
\end{prop}

\par

Because of Proposition \ref{p1.4}{\,}(1) we are allowed to be rather imprecise concerning
the choice of $\phi \in  M^r_{(v)}\setminus 0$ in
\eqref{modnorm}. For instance let $C>0$ be a constant and $\Omega$ be a
subset of $\Sigma _1'$. If we then write, that $\nm a{M^{p,q}_{(\omega )}}\le C$ for
every $a\in \Omega$,  we mean  that the inequality holds for some choice
of $\phi \in  M^r_{(v)}\setminus 0$ and every $a\in
\Omega$.
Additionally a similar inequality is true for any other choice
of $\phi \in  M^r_{(v)}\setminus 0$, 
although we may have to  replace $C$ by another constant.

\par

We refer to \cite {Fe4,FG1,FG2,FG4,GaSa,Gc2,RSTT,To20}
for more facts about modulation spaces.

\par

\subsection{A broader family of modulation spaces}

\par
In this subsection we introduce a broader class of modulation spaces, by
imposing certain types of translation invariant solid BF-space norms on the short-time
Fourier transforms, cf. \cite{Fe4,Fe6,Fe8,FG1,FG2}.

\par

We recall that a quasi-norm $\| \cdo \|_{\mascB}$ of order $r \in (0,1]$ on the
vector-space $\mascB$ is a nonnegative functional on $\mascB$ which satisfies
\begin{alignat}{2}
 \nm {f+g}{\mascB} &\le 2^{\frac 1r-1}(\nm {f}{\mascB} + \nm {g}{\mascB}), &
\quad f,g &\in \mascB ,
\label{Eq:WeakTriangle1}
\\[1ex]
\| \alpha \cdot f \|_{\mascB} &= |\alpha| \cdot \|f\|_{\mascB},
& \quad \alpha &\in \mathbf{C},
\quad  f \in \mascB
\notag
\intertext{and}
   \|f\|_{\mascB} &= 0\quad  \Leftrightarrow \quad f=0. & &
\notag
\end{alignat}

\par

The vector space $\mascB$ is called a quasi-Banach space if it is a complete quasi-normed space.
If $\mascB$ is a quasi-Banach space with quasi-norm satisfying \eqref{Eq:WeakTriangle1}
then on account of \cite{Aik,Rol} there is an equivalent quasi-norm to $\nm \cdo {\mascB}$
which additionally satisfies
\begin{align}\label{Eq:WeakTriangle2}
\nm {f+g}{\mascB}^r \le \nm {f}{\mascB}^r + \nm {g}{\mascB}^r, 
\quad f,g \in \mascB .
\end{align}
From now on we always assume that the quasi-norm of the quasi-Banach space $\mascB$
is chosen in such way that both \eqref{Eq:WeakTriangle1} and \eqref{Eq:WeakTriangle2}
hold.

\par

\begin{defn}\label{bfspaces1}
Let $\mascB \subseteq L^r_{loc}(\rr d)$ be a quasi-Banach space
of order $r\in (0,1]$ which contains $\Sigma _1(\rr d)$ with continuous embedding,
and let $v _0\in\mascP _E(\rr d)$.
Then $\mascB$ is called a \emph{translation invariant
Quasi-Banach Function space on $\rr d$} (with respect to $v$), or \emph{invariant
QBF space on $\rr d$}, if there is a constant $C$ such
that the following conditions are fulfilled:
\begin{enumerate}
\item if $x\in \rr d$ and $f\in \mascB$, then $f(\cdo -x)\in
\mascB$, and 
\begin{equation}\label{translmultprop1}
\nm {f(\cdo -x)}{\mascB}\le Cv_0(x)\nm {f}{\mascB}\text ;
\end{equation}

\vrum

\item if  $f,g\in L^r_{loc}(\rr d)$ satisfy $g\in \mascB$ and $|f|
\le |g|$, then $f\in \mascB$ and
$$
\nm f{\mascB}\le C\nm g{\mascB}\text .
$$
\end{enumerate}
\end{defn}

\par

If the weight $v$ even is an element of $\mascP _{E,s}(\rr d)$
($\mascP _{E,s}^0(\rr d)$), then we call $\mascB$ of Definition \ref{bfspaces1}
an \emph{invariant QBF-space of Roumieu type (Beurling type)} of order $s$.

\par

By means of (2) in Definition \ref{bfspaces1} we know that $f\cdot h\in \mascB$ if $f\in
\mascB$ and $h\in L^\infty$ and additionally
\begin{equation}\label{multprop}
\nm {f\cdot h}{\mascB}\le C\nm f{\mascB}\nm h{L^\infty}.
\end{equation}
For $r=1$, the invariant QBF space $\mascB$ of Definition \ref{bfspaces1} becomes a Banach
space and  is called an
\emph{invariant BF-space} (with respect to $v$). Because of condition (2) a
translation invariant BF-space is a solid BF-space in the sense of
(A.3) in \cite{Fe6}. 
For each invariant BF-space $\mascB \subseteq L^1_{loc}(\rr d)$ we have
Minkowski's inequality, i.{\,}e.
\begin{equation}\label{Eq:MinkIneq}
\nm {f*\fy}{\mascB}\le C \nm {f}{\mascB}\nm \fy{L^1_{(v)}},
\qquad f\in \mascB ,\ \fy \in \Sigma _1 (\rr d)
\end{equation}
for some $C>0$ which is independent of
$f\in \mascB$ and $\fy \in \Sigma _1 (\rr d)$. The density of $\Sigma _1$
in $L^1_{(v)}$ provides that the definition of
$f*\fy$ extends uniquely to any $f\in \mascB$ and
$\fy \in L^1_{(v)}(\rr d)$. Hence \eqref{Eq:MinkIneq} is also true
for such $f$ and $\fy$.

\par

The following result shows that $v_0$ in Definition \ref{bfspaces1} can be replaced
by a submultiplicative weight $v$ such that \eqref{translmultprop1} is true with
$v$ in place of $v_0$ and the constant $C=1$, and such that
\begin{equation}\label{Eq:StrongSubMult}
v(x+y)\le v(x)v(y)
\quad \text{and}\quad
v(-x)=v(x),\qquad x,y\in \rr d.
\end{equation}

\par

\begin{prop}
Let $\mascB$ be an invariant BF-space on $\rr d$ with respect to
$v_0\in \mascP _E(\rr d)$.
Then there is a $v\in \mascP _E(\rr d)$ which satisfies \eqref{Eq:StrongSubMult}
and such that \eqref{translmultprop1} holds with $v$ in place of $v_0$, and 
$C=1$.
\end{prop}

\par

\begin{proof}
Let
$$
v_1(x) \equiv
\sup _{f\in \mascB}
\left (
\frac {\nm {f(\cdo -x)}{\mascB}}{\nm f{\mascB}}
\right ) .
$$
Then
\begin{multline*}
v_1(x+y) = \sup _{f\in \mascB}
\left (
\frac {\nm {f(\cdo -x-y)}{\mascB}}{\nm {f(\cdo -y)}{\mascB}} 
\cdot
\frac {\nm {f(\cdo -y)}{\mascB}}{\nm f{\mascB}} 
\right )
\\[1ex]
\le
\sup _{f\in \mascB}
\left (
\frac {\nm {f(\cdo -x)}{\mascB}}{\nm f{\mascB}}
\right ) 
\cdot
\sup _{f\in \mascB}
\left (
\frac {\nm {f(\cdo -y)}{\mascB}}{\nm f{\mascB}}
\right ) = v_1(x)v_1(y).
\end{multline*}
The result now follows by letting
$$
v(x)=\max (v_1(x),v_1(-x)). \text{\phantom k \hspace{\stretch{1}}   \qedhere}
$$
\end{proof}

\par

From now on it is assumed that $v$ and $v_j$ are submutliplicative weights
if nothing else is stated.

\par

\begin{example}\label{Lpqbfspaces}
For $p,q\in [1,\infty ]$ the space $L^{p,q}_1(\rr {2d})$ consists of 
all $f\in L^1_{loc}(\rr {2d})$ such that
$$
\nm  f{L^{p,q}_1} \equiv \Big ( \int \Big ( \int |f(x,\xi )|^p\, dx\Big
)^{q/p}\, d\xi \Big )^{1/q} < \infty.
$$
Additionally $L^{p,q}_2(\rr {2d})$ is the set of all $f\in
L^1_{loc}(\rr {2d})$ such that
$$
\nm  f{L^{p,q}_2} \equiv \Big ( \int \Big ( \int |f(x,\xi )|^q\, d\xi
\Big )^{p/q}\, dx \Big )^{1/p} < \infty.
$$
Then $L^{p,q}_1$ and $L^{p,q}_2$ are
translation invariant BF-spaces with respect to $v=1$.
\end{example}

\par

For translation invariant BF-spaces we make the
following observation.

\par

\begin{prop}\label{p1.4BFA}
Assume that $v\in\mascP _E(\rr {d})$, and that $\mascB$ is an
invariant BF-space with respect to $v$ such that \eqref{Eq:MinkIneq}
holds true. Then the
convolution mapping $(\fy ,f)\mapsto \fy *f$ from $C_0^\infty (\rr
d)\times \mascB$ to $\mascB$ extends uniquely to a continuous
mapping from
$L^1_{(v )}(\rr d)\times \mascB$ to $\mascB$, and \eqref{Eq:MinkIneq}
holds true for any $f\in \mascB$ and $\fy \in L^1_{(v)}(\rr d)$.
\end{prop}

\par

The result is a straight-forward consequence of the fact that $C_0^\infty$
is dense in $L^1_{(v)}$.

\par

Next we define the extended class of modulation spaces, which are of interest for us:

\par

\begin{defn}\label{bfspaces2}
Assume that $\mascB$ is a translation
invariant QBF-space on $\rr {2d}$, $\omega \in\mascP _E(\rr {2d})$,
and that $\phi \in
\Sigma _1(\rr d)\setminus 0$. Then the set $M(\omega ,\mascB )$ consists
of all $f\in \Sigma _1'(\rr d)$ such that
$$
\nm f{M(\omega ,\mascB )}
\equiv \nm {V_\phi f\, \omega }{\mascB}
$$
is finite.
\end{defn}

\par

Obviously, we have
$
M^{p,q}_{(\omega )}(\rr d)=M(\omega ,\mascB )$
if  $\mascB =L^{p,q}_1(\rr {2d})$, see e.g. (\ref{Lpqbfspaces}).
We remark, that  many properties of the classical modulation
spaces  are also true for $M(\omega ,\mascB )$.
For instance,
the definition of $M(\omega ,\mascB )$ is independent of the
choice of $\phi$ when $\mascB$ is a Banach space. This statement is
formulated in the next proposition. It can be proved by  similar arguments
as Proposition 11.3.2 in \cite{Gc2}. Hence we omit the proof.

\par

\begin{prop}\label{p1.4BF}
Let $\mascB$ be an invariant BF-space with
respect to $v_0\in \mascP _E(\rr {2d})$. Also let
$\omega ,v\in\mascP _E(\rr {2d})$ be such that $\omega$ is
$v$-moderate, $M(\omega ,\mascB )$ is the same as in Definition
\ref{bfspaces2}, and let $\phi \in M^1_{(v_0v)}(\rr d)\setminus
0$ and $f\in \Sigma _1'(\rr d)$. Then $f\in M(\omega ,\mascB )$
if and only if $V_\phi f\, \omega \in \mascB$, and
different choices of $\phi$ gives rise to equivalent norms in
$M(\omega ,\mascB )$.
\end{prop}

\par

In applications, the quasi-Banach space $\mascB$ is mostly a
mixed quasi-normed Lebesgue space, which is defined next. 
Let $E= \{ e_1,\dots,e_d \}$ be an orderd basis of $\rr d$ and let $E'=\{ e'_1,\dots,e'_d \}$ be such that 
$$
\scal {e_j} {e'_k} = 2\pi \delta_{jk},
\quad
j,k =1,\dots, d.
$$
Then $E'$ is called the dual basis of $E$. The corresponding lattice and dual lattice are 
\begin{align*}
\Lambda _E &=\sets{j_1e_1+\cdots +j_de_d}{(j_1,\dots,j_d)\in \zz d},
\\[1ex]
\intertext{and}
\Lambda'_E &= \Lambda_{E'}=\sets{\iota _1e'_1+\cdots +\iota _de'_d}
{(\iota _1,\dots ,\iota _d) \in \zz d},
\end{align*}

\par

There is a matrix $T_E$ such that
$e_1,\dots ,e_d$ and $e_1',\dots ,e_d'$ are the images of
the standard basis under $T_E$ and
$T_{E'}= 2\pi(T^{-1}_E)^t$, respectively. We also let $\kappa (E)$ be
the parallelepiped spanned by the basis $E$.

\par

We define for each $\mabfq =(q_1,\dots ,q_d)\in (0,\infty ]^d$
$$
\max \mabfq =\max (q_1,\dots ,q_d)
\quad \text{and}\quad
\min \mabfq =\min (q_1,\dots ,q_d).
$$

\par

\begin{defn}\label{Def:MixedLebSpaces}
Let $E = \{ e_1,\dots ,e_d\}$ be an orderd basis of $\rr d$, $\omega$ be a weight on
$\rr d$,
$\mabfp =(p_1,\dots ,p_d)\in (0,\infty ]^{d}$ and $r=\min (1,\mabfp )$.
If  $f\in L^r_{loc}(\rr d)$, then
$$
\nm f{L^{\mabfp }_{E,(\omega )}}\equiv
\nm {g_{d-1}}{L^{p_{d}}(\mathbf R)},
$$
where  $g_k(z_k)$, $z_k\in \rr {d-k}$,
$k=0,\dots ,d-1$, are inductively defined as
\begin{align*}
g_0(x_1,\dots ,x_{d}) &\equiv |f(x_1e_1+\cdots +x_{d}e_d)
\omega (x_1e_1+\cdots +x_{d}e_d)|,
\\[1ex]
\intertext{and}
g_k(z_k) &\equiv
\nm {g_{k-1}(\cdo ,z_k)}{L^{p_k}(\mathbf R)},
\quad k=1,\dots ,d-1.
\end{align*}
The space $L^{\mabfp }_{E,(\omega )}(\rr d)$ consists
of all $f\in L^r_{loc}(\rr d)$ such that
$\nm f{L^{\mabfp}_{E,(\omega )}}$ is finite, and is called
\emph{$E$-split Lebesgue space (with respect to $\mabfp$ and $\omega$}).
\end{defn}

\par

Let $E$, $\mabfp$ and $\omega$ be the same as in Definition \ref{Def:MixedLebSpaces}.
Then the discrete version $\ell ^{\mabfp}_{E,(\omega )}(\Lambda _E)$
of $L^{\mabfp}_{E,(\omega )}(\rr d)$ is the set of all
sequences $a=\{ a(j)\} _{j\in \Lambda _E}$ such that the quasi-norm
$$
\nm a{\ell ^{\mabfp}_{E,(\omega )}}
\equiv
\nm {f_a}{L^{\mabfp}_{E,(\omega )}},\qquad f_a=\sum _{j\in \Lambda _E} a(j)\chi _j,
$$
is finite. Here $\chi _j$ is the characteristic function of $j+\kappa (E)$.
We also set $L^{\mabfp}_{E} = L^{\mabfp}_{E,(\omega )}$
and $\ell ^{\mabfp}_{E} = \ell ^{\mabfp}_{E,(\omega )}$
when $\omega =1$.

\par

\begin{defn}\label{Def:MixedPhaseShiftLebSpaces}
Let $E$ be an ordered basis of the phase space $\rr {2d}$. Then
$E$ is called \emph{phase split}
if there is a subset $E_0\subseteq E$ such that the span of $E_0$ equals
$\sets {(x,0)\in \rr {2d}}{x\in \rr d}$, and the span of $E\setminus E_0$
equals $\sets {(0,\xi )\in \rr {2d}}{\xi \in \rr d}$.
\end{defn}

\par

%

\par

\subsection{Pilipovi{\'c} flat spaces, modulation spaces outside
time-frequency analysis and the Bargmann transform}

\par
Besides the characterization by means of the short-time Fourier transform,
see Proposition \ref{stftGelfand2}, 
Gelfand-Shilov spaces can be characterized via Hermite function
expansion, too. Here the Hermite function of order $\alpha \in \nn d$ is given by
$$
h_\alpha (x) = \pi ^{-\frac d4}(-1)^{|\alpha |}
(2^{|\alpha |}\alpha !)^{-\frac 12}e^{\frac {|x|^2}2}
(\partial ^\alpha e^{-|x|^2}).
$$
We also can write $h_{\alpha}$ via
$$
h_{\alpha}(x)=   ( (2\pi )^{\frac d2} \alpha ! )^{-1}
e^{-\frac {|x|^2}2}p_{\alpha}(x),
$$
for some polynomial $p_\alpha$ on $\rr d$. $p_\alpha$ are
called the Hermite polynomial of order $\alpha$. It is well-known
that $\{ h_\alpha \} _{\alpha \in \nn d}$ provides an orthonormal basis
for $L^2(\rr d)$.

\par

We now can characterize the Gelfand-Shilov spaces $\maclS _s(\rr d)$
($\Sigma _s(\rr d)$) with $s\ge \frac 12$ ($s>\frac 12$) as follows:
$f\in \maclS _s(\rr d)$ ($f\in \Sigma _s(\rr d)$), if and only
if the coefficients $c_\alpha (a)$ in its Hermite series expansion
\begin{equation}\label{Eq:HermiteExp}
f = \sum _{\alpha \in \nn d}c_\alpha (f)h_\alpha ,\qquad
c_\alpha (f) = (f,h_\alpha ) 
\end{equation}
fulfilles 
$$
|c_\alpha (f)| \lesssim e^{-r|\alpha |^{\frac 1{2s}}},
$$
for some $r>0$ (for every $r>0$). 
Various kinds of Fourier-invariant functions and
distribution spaces can be obtained by applying suitable topologies
on formal power series expansions, cf. e.{\,}g. \cite{FeGaTo1,To22}.
To mention one of them, which is of peculiar interest: The \emph{Pilipovi{\'c} flat
space}, $\maclH _\flat
(\rr d)$ and its dual $\maclH _\flat '(\rr d)$,
are defined by all formal expansions \eqref{Eq:HermiteExp} such that
$$
|c_\alpha (f)| \lesssim r^{|\alpha |}\alpha !^{-\frac 12}
$$
for some $r>0$, respectively
$$
|c_\alpha (f)| \lesssim r^{|\alpha |}\alpha !^{\frac 12}
$$
for every $r>0$. For $f\in \maclH
_\flat '(\rr d)$ and $\phi \in \maclH _\flat (\rr d)$,
we define
$$
(f,\phi )_{L^2(\rr d)} \equiv \sum _{\alpha \in \nn d}
c_\alpha (f)\overline{c_\alpha (\phi)}.
$$
If $\phi, f \in L^2(\rr d)$, the pairing $(f,\phi )_{L^2(\rr d)}$ agree's with the
$L^2(\rr d)$ scalar product of those two functions.

\par

We remark that $\maclH _\flat '(\rr d)$ is larger than any
Fourier-invariant Gelfand-Shilov distribution space,
and $\maclH _\flat (\rr d)$ is
smaller than any Fourier-invariant Gelfand-Shilov space.
We already know, that any
$f\in \maclS _s(\rr d)$ ($f\in \Sigma _s(\rr d)$) with $s\ge \frac 12$ ($s> \frac 12$) can be
expressed in a unique way by an expansion \eqref{Eq:HermiteExp}
with convergence in $\maclS _s'(\rr d)$ ($\Sigma _s'(\rr d)$).
Similarly, we have $f\in \maclS' _s(\rr d)$ ($f\in \Sigma' _s(\rr d)$),
if and only if
$$
|c_\alpha (f)| \lesssim e^{r|\alpha |^{\frac 1{2s}}},
$$
for every $r>0$ (for some $r>0$).

\par

One reason, why the Pilipovi{\'c} flat space $\maclH _\flat
(\rr d)$ and its dual $\maclH _\flat '(\rr d)$ are of particular interest, are is their images
under the Bargmann transform. The kernel of the Bargmann transform
is given by
$$
\mathfrak A_d(z,y)=\pi ^{-\frac d4} \exp \Big ( -\frac 12(\scal
zz+|y|^2)+2^{\frac 12}\scal zy\Big ),
$$
which is analytic in $z$. Seen as a function of $y$, $\mathfrak A_d$ belongs to $\maclH _\flat
(\rr d)$. 
We define the Bargmann transform $(\mathfrak V_df)(z)$ of a function $f\in \maclH _\flat
'(\rr d)$ by
$$
(\mathfrak V_df)(z) =\scal f{\mathfrak A_d(z,\cdo )},
$$
where $\scal f\phi = (f,\overline \phi )_{L^2(\rr d)}$.
Due to \cite{To22} we know that $\mathfrak V_d$ is bijective between
$\maclH _\flat '(\rr d)$ and $A(\cc d)$, the set of all entire functions
on $\cc d$, and restricts to a bijective map from $\maclH _\flat (\rr d)$
and
$$
\sets {F\in A(\cc d)}{|F(z)|\lesssim e^{R|z|},\ \text{for some}\ R>0} .
$$

\par

Later on we will need,  that the Bargmann
and the short-time Fourier transform are linked  by the formula
\begin{equation}\label{Eq:BargmannSTFTlink}
\begin{aligned}
(\mathfrak{V} _df)(x+\im \xi ) 
&=  (2\pi )^{\frac d2}e^{\frac 12(|x|^2+|\xi|^2)}e^{-i\scal x\xi}
(V_\phi )f(2^{\frac 12}x,-2^{\frac 12}\xi ),
\\[1ex]
\phi (x) &= \pi ^{-\frac d4}e^{-\frac 12|x|^2},\qquad x\in \rr d,
\end{aligned}
\end{equation}
This can be shown by straight-forward computations. By means of
the operator
\begin{equation}\label{EqUBDef}
(U_{\mathfrak V}F)(x,\xi ) =
(2\pi )^{\frac d2}e^{\frac 12(|x|^2+|\xi|^2)}e^{-i\scal x\xi}
F(2^{\frac 12}x,-2^{\frac 12}\xi ),
\end{equation}
when $F$ is a function or a suitable element of $F \in \maclH' _\flat
(\rr d)$
we can write the Bargmann transform as
$$
(\mathfrak{V} _df)(x+\im \xi )
= (U_{\mathfrak V}(V_\phi f))(x,\xi ).
$$

\par

\begin{defn}\label{Lemma:TheSpaces}
Let $\phi$ be as in \eqref{Eq:BargmannSTFTlink},
$\omega$ be a weight on $\rr {2d}$, $\mascB $ be an invariant
QBF-space with respect to $v\in \mascP _E(\rr {2d})$
on $\rr {2d}\simeq \cc d$ of order $r\in (0,1]$.
\begin{enumerate}
\item $B(\omega ,\mascB )$ consists of all
$F\in L^r_{loc}(\rr {2d})= L^r_{loc}(\cc {d})$ such that
$$
\nm F{B(\omega ,\mascB )}\equiv
\nm {(U_{\mathfrak V}^{-1}F)\omega }{\mascB }<\infty .
$$
Here $U_{\mathfrak V}$ is given by \eqref{EqUBDef};

\vrum

\item $A(\omega ,\mascB )$ consists of all $F\in A(\cc
d)\cap B(\omega ,\mascB )$ with topology inherited
from $B(\omega ,\mascB )$;

\vrum

\item $M(\omega ,\mascB )$ consists of all $f\in
\maclH _\flat' (\rr d)$ such that
$$
\nm f{M(\omega ,\mascB)} \equiv
\nm {V_\phi f \cdot \omega}{\mascB}
$$
is finite.
\end{enumerate}
\end{defn}

\par

We observe the small restrictions on $\omega$ compared to
what is the main stream, e.{\,}g. that $\omega$ should belong
to $\mascP _E(\rr {2d})$ or be moderated by functions which are
bounded by polynomials. We still call the space $M(\omega ,\mascB )$
as the modulation space with respect to $\omega$ and $\mascB$.
In contrast to earlier situations, it seems that $M(\omega ,\mascB )$
is not invariant under the choice of $\phi$ when $\omega$ fails to
belong to $\mascP _E$. For that reason we always assume that the
weight function is given by \eqref{Eq:BargmannSTFTlink} for such
$\omega$.

\par

We have the following.

\par

\begin{prop}\label{Prop:MadAnalIdent}
Let $\phi$ be as in \eqref{Eq:BargmannSTFTlink}, $\omega$
be a weight on $\rr {2d}$, and let $\mascB $ be an invariant
QBF-space with respect to $v\in \mascP _E(\rr {2d})$. Then
the following is true:
\begin{enumerate}
\item the map $\mathfrak V _d$ is an isometric bijection from
$M(\omega ,\mascB )$ to $A(\omega ,\mascB )$;

\vrum

\item if in addition $\mascB$ is a mixed quasi-norm space
of Lebesgue types, then $M(\omega ,\mascB )$ and $A(\omega
,\mascB )$
are quasi-Banach spaces, which are Banach spaces in the case
$\mascB$ is a Banach space.
\end{enumerate}
\end{prop}

\par

\begin{proof}
From \eqref{Eq:BargmannSTFTlink}, \eqref{EqUBDef}
and Definition \ref{Lemma:TheSpaces} it follows that
$\mathfrak V _d$ is an isometric injection from
$M(\omega ,\mascB )$ to $A(\omega ,\mascB )$. Since
any element in $A(\cc d)$, and thereby any element
in $A(\omega ,\mascB)$ is a Bargmann transform
of an element in $\maclH' _\flat (\rr d)$, it follows
that the image of $M(\omega ,\mascB )$ under
$\mathfrak V_d$ contains $A(\omega ,\mascB )$. This gives
the stated bijectivity in (1).

\par

The completeness of $A(\omega ,\mascB )$, and thereby of
$M(\omega ,\mascB )$ follows from \cite{To22}. The details
are left for the reader.
\end{proof}

\par

\section{Compactness properties for modulation
spaces}\label{sec2}

\par

This section is devoted to the questions under which sufficient and necessary
conditions the inclusion map 
\begin{align*}
\iota : M(\omega _1, \mathscr{B}) \rightarrow M(\omega _2, \mathscr{B})
\end{align*}
is continuous or even compact for suitable invariant QBF-spaces $\mathscr{B}$.

\par

As ingredients for the proof of our main results we need to deduce  some
properties for moderate weight functions. In what follows let $L^{\infty}_{0,(\omega )}(\rd)$
be the set of all $f \in L^{\infty}_{(\omega )}(\rd)$ with the property
\begin{align*}
  \lim_{R \rightarrow \infty} \left( \essup_{|x| \geq R} |f(x)\omega (x)| \right) =0,
\end{align*}
when $\omega$ is a weight on $\rr d$. We also set $L^{\infty}_{0}=L^{\infty}_{0,(\omega )}$
when $\omega =1$. If $\Lambda$ is a lattice, then the discrete Lebesgue spaces
$\ell ^\infty _0(\Lambda )$ and $\ell ^\infty _{0,(\omega )}(\Lambda )$ are defined analogously.

\par

\begin{lemma}\label{compares1}
Let $E$ be an ordered basis of $\rr d$ and let $\omega \in \mascP _E(\rr d)$.
Then the following is true:
\begin{enumerate}
\item $\mascP _E(\rr d)$ is a convex cone which is closed under
multiplication, division and under compositions with power functions;

\vrum

\item $\mascP _E(\rr {2d})\cap L^{\mabfp}_{E,(\omega )}(\rr {d})$
increases with $\mabfp \in (0,\infty ]^d$, and
\begin{equation}\label{compares2}
\mascP _E(\rr {d})\cap L^{\mabfp}_{E,(\omega )}(\rr {d})
\subseteq
\mascP _E(\rr {d})\cap L^\infty _{0,(\omega )}(\rr {d}),
\qquad \mabfp \in (0,\infty )^d.
\end{equation}
\end{enumerate}
\end{lemma}

\par

Similar properties has already been shown in \cite[Lemma 2.1]{BoTo} for the
smaller weight space $\mascP$.

\begin{proof} Claim (1) can easily be verified by means of the definition of moderate weights. 

\par

It remains to verify (2). Let $\kappa (E)$ be the (closed) parallelepiped spanned by $E$,
and let $\vartheta \in \mascP _E(\rr d)$. By using the map $\vartheta \mapsto \vartheta
\cdot \omega$, we reduce ourself to the case when $\omega =1$.

The moderateness of $\vartheta \in
\mascP _E(\rr {2d})$ implies that
\begin{equation}\label{ineq}
\vartheta (x_1)\asymp  \vartheta (x_1+x_2)
\quad \text{when}\quad
x_2\in \kappa (E).
\end{equation}
Hence, if $\chi _j$ is the characteristic function of $j+\kappa (E)$, and
$$
\vartheta _0 (x) = \sum _{j\in \Lambda _E}\vartheta (j)\chi _j(x),
$$
then $\vartheta \asymp \vartheta _0$, giving that
$$
\nm \vartheta{L^{\mabfp}_E} \asymp \nm {\vartheta _0}{L^{\mabfp}_E}
\asymp \nm \vartheta {\ell ^{\mabfp}_E}.
$$
The assertion now follows from the fact that $\ell ^{\mabfp}_E$ increases with
$\mabfp$ and that if in addition $\mabfp \in (0,\infty )^d$, then
$\ell ^{\mabfp}_E\subseteq \ell ^\infty _0$.
\end{proof}

\par

We also have the following result, which is an immediate consequence
of \cite[Theorem 2.5]{To26}.

\par

\begin{prop}\label{Prop:LargestModSpace}
Let $v,v_0\in \mascP _E(\rr {2d})$ be
submultiplicative, $\omega \in \mascP _E(\rr {2d})$ be $v$-moderate,
and let $\mascB$ be an invariant BF-space with respect to $v_0$. Then
$M(\omega ,\mascB)$ is a Banach space, and
\begin{equation}\label{Eq:LargestModSpace}
M(\omega ,\mascB )\hookrightarrow M^\infty
_{(1/(v_0v))}(\rr d) .
\end{equation}
\end{prop}

\par

\begin{rem}\label{Rem:LargestModSpace}
If $\mascB =L^{\mabfp}_E(\rr {2d})$  for some phase split basis $E$
of $\rr {2d}$, $\mabfp \in (0,\infty ]^{2d}$ and $\omega \in \mascP _E(\rr {2d})$,
then $M(\omega ,\mascB )$ is a quasi-Banach space. Moreover
$M(\omega , L^{\mabfp}_E(\rr {2d}))$ is increasing with $\mabfp$. In particular,
\eqref{Eq:LargestModSpace}  is improved into
$$
M(\omega , L^{\mabfp}_E(\rr {2d}))\hookrightarrow M^\infty _{(\omega )}(\rr {2d}).
$$
We refer to  \cite{To20} for the proof.
\end{rem}

\par

For the proof  the twisted convolution $\widehat *$ of two functions $F,G\in L^1(\rr {2d})$
defined by
$$
(F{\, \widehat *\,}G)(x,\xi ) = (2\pi )^{-\frac d2}
\iint _{\rr {2d}}F(x-y,\xi -\eta )G(y,\eta )
e^{-i\scal {x-y}\eta}\, dyd\eta ,
$$
is needed. The twisted convolution is continuous as a map between several
function spaces, see e.{\,}g. \cite{Gc1} or Lemma 3 in \cite{CoJoTo1}. For
instance the map $(F,G)\mapsto F{\, \widehat *\,}G$
is continuous from $L^1(\rr {2d})\times L^1(\rr {2d})$ to
$L^1(\rr {2d})$, and can be restricted to a
continuous map from 
$\Sigma _1 (\rr {2d})\times \Sigma _1 (\rr {2d})$ to
$\Sigma _1(\rr {2d})$. The latter map can be continuously extended 
to a continuous map from
$\Sigma _1 '(\rr {2d})\times \Sigma _1 (\rr {2d})$ to
$\Sigma _1'(\rr {2d})$.

\par

On account of  the Fourier's inversion formula we obtain for all $f\in
\Sigma _1'(\rr d)$ and $\phi _1,\phi _2,\phi _3
\in \Sigma _1(\rr d)$:
\begin{equation}\label{Eq:TwistConvSTFT}
(\phi _3,\phi _1)_{L^2}\cdot V_{\phi _2}f = (V_{\phi _1}f){\, \widehat *\,}
(V_{\phi _2}\phi _3).
\end{equation}
Since $(\phi _2,\phi _3)\mapsto
V_{\phi _2}\phi _3$ is continuous from $\Sigma _1 (\rr {d})
\times \Sigma _1 (\rr {d})$ to $\Sigma _1(\rr {2d})$ we get for $\phi _1=\phi _2
=\phi _3=\phi \in \Sigma _1(\rr d)\setminus 0$ the continuity of
the operator $P_\phi$, defined by
\begin{equation}\label{Eq:ProjOp}
P_\phi F \equiv \nm {\phi}{L^2(\rr d)}^{-2} F{\, \widehat *\,}(V_\phi \phi)
\end{equation}
on $\Sigma _1'(\rr {2d})$. The operator $P_\phi$ has the following properties:

\par

\begin{lemma}\label{Lemma:Proj}
Let $\phi \in \Sigma _1(\rr d)$. Then the following is true:
\begin{enumerate}
\item $P_\phi$ in \eqref{Eq:ProjOp} is a continuous projection from
$\Sigma _1'(\rr {2d})$ to
$$
V_\phi (\Sigma _1'(\rr d))
\equiv
\sets {V_\phi f}{f\in \Sigma _1'(\rr d)}
\subseteq \Sigma _1'(\rr {2d})\bigcap C^\infty (\rr {2d}) \text ;
$$

\vrum

\item $P_\phi$ in \eqref{Eq:ProjOp} restricts to a continuous
projection from $\Sigma _1(\rr {2d})$ to
$$
V_\phi (\Sigma _1(\rr d))
\equiv
\sets {V_\phi f}{f\in \Sigma _1(\rr d)}
\subseteq \Sigma _1(\rr {2d}) \text ;
$$

\vrum

\item if $\mascB$ is an invariant BF-space on $\rr {2d}$, then
$P_\phi$ is continuous on $\mascB$.
\end{enumerate}
\end{lemma}

\par

Related results can essentially be found in e.{\,}g. \cite{Gc1,Fo1}.
In order to be self-contained, we here give a short proof.

\par

\begin{proof}
By \eqref{Eq:TwistConvSTFT} it is clear that $P_\phi$ is the identity
map on $V_\phi (\Sigma _1'(\rr d))$ and thereby on $V_\phi (\Sigma _1(\rr d))$.

\par

Let $V_\phi ^*$ be the $L^2$-adjoint of $V_\phi$. That is,
$V_\phi ^*F$ satisfies
$$
(V_\phi ^*F,\psi )_{L^2(\rr d)} = (F,V_\phi \psi )_{L^2(\rr {2d})},
\qquad F\in \Sigma _1'(\rr {2d}),\ \psi \in \Sigma _1(\rr d).
$$
By the continuity properties of $V_\phi$ on $\Sigma _1$ and $\Sigma _1'$,
it follows that $V_\phi ^*$ is continuous from $\Sigma _1'(\rr {2d})$
to $\Sigma _1'(\rr d)$ and restricts to a continuous map
from $\Sigma _1(\rr {2d})$ to $\Sigma _1(\rr d)$.

\par

By a straight-forward application of Fourier's inversion formula
it follows that
$$
P_\phi F = V_\phi f
\qquad \text{when}\quad f= \nm \phi{L^2}^{-2}V_\phi ^*F,
$$
which shows that the images of $\Sigma _1'(\rr {2d})$
and $\Sigma _1(\rr {2d})$ under $P_\phi $ equals
$V_\phi (\Sigma _1'(\rr d))$ and $V_\phi (\Sigma _1(\rr d))$,
respectively. This gives (1) and (2).

\par

If $\mascB$ is an invariant BF-space on $\rr {2d}$ and $F\in \mascB$,
then it follows from the definitions that
$$
|P_\phi F |\lesssim |F|*\Phi ,
$$
where $\Phi = |V_\phi \phi|$ belongs to $L^1_{(v)}(\rr {2d})$ for every choice
of $v\in \mascP _E(\rr {2d})$. Hence, a combination of (2) in Definition
\ref{bfspaces1} and \eqref{Eq:MinkIneq} gives
$P_\phi F \in \mascB$, and
$$
\nm {P_\phi F}{\mascB} \lesssim \nm F{\mascB}\nm \Phi{L^1_{(v)}},
$$
for some $v\in \mascP _E(\rr {2d})$, and the continuity of $P_\phi$
on $\mascB$ follows. This gives (3).
\end{proof}

\par

\begin{lemma}\label{Lemma:GSembBF}
If $\mascB$ is an invariant BF-space of $\rr d$, then
$$
L^\infty _{(v)}(\rr d)\hookrightarrow \mascB
$$
for some $v\in \mascP _E(\rr d)$. Then 
\end{lemma}

\par

\begin{proof}
Since $\Sigma _1(\rr d)$ is continuously embedded in $\mascB$ we have
$$
\nm f{\mascB} \lesssim \sup _{\beta \in \nn d} 
\left (
\frac {\nm {(D^\beta f) \cdot e^{|\cdo |/h_0}}{L^\infty}}{h^{|\beta |}_0\beta !}
\right )
$$
for some $h_0>0$. Let $\omega = e^{-2|\cdo |/h_0}$, $\omega _0
=\omega *e^{-|\cdo |^2/2}$ and let $v=1/\omega _0$. Then
\cite[Proposition 1.6]{AbCoTo} shows that 
$$
|D^\beta \omega _0| \lesssim h^{|\beta |}\beta !\, e^{-2|\cdo |/h_0}
$$
for every $h>0$. By choosing $h<h_0$ we get 
\begin{multline*}
\nm {\omega _0}{\mascB}
\lesssim
\sup _{\beta \in \nn d}
\left (
\frac {\nm {D^\beta \omega _0 \cdot e^{|\cdo |/h_0}}{L^\infty}}
{h_0^{|\beta |}\beta !}
\right )
\lesssim
\sup _{\beta \in \nn d}
\left (
\frac {h^{|\beta |}\beta !\nm {\omega _0 \cdot e^{|\cdo |/h_0}}{L^\infty}}
{h_0^{|\beta |}\beta !}
\right )
\\[1ex]
\lesssim
\nm {e^{-2|\cdo |/h_0}\cdot e^{|\cdo |/h_0}}{L^\infty} =1<\infty .
\end{multline*}

\par

Hence, if $f\in L^\infty _{(v)}(\rr d)$, then
$$
\nm f{\mascB} \lesssim \nm {\omega _0}{\mascB}\nm {f\cdot v}{L^\infty}
\asymp  \nm f{L^\infty _{(v)}},
$$
and the result follows.
\end{proof}

\par

\begin{lemma}
Let $\mascB$ be an invariant BF-space on $\rr d$
and $\omega \in \mascP _E(\rr d)$. Then 
$$
\mascB _{(\omega )} \equiv \sets {f\in L^1_{loc}(\rr d)}{f\cdot \omega \in \mascB}
$$
is an invariant BF-space under the norm
$$
f\mapsto \nm f{\mascB _{(\omega )}}
\equiv
\nm {f\cdot \omega}{\mascB}.
$$
\end{lemma}

\par

\begin{proof}
Let $v$ be as in .
By Lemma \ref{Lemma:GSembBF}, $L^\infty _{(\omega \cdot v)}
\hookrightarrow \mascB _{(\omega )}$. Since
$$
\Sigma _1(\rr d)\hookrightarrow L^\infty _{(\omega \cdot v)}(\rr d),
$$
it follows that $\Sigma _1(\rr d)$ is continuously embedded in $\mascB _{(\omega )}$.

\par

By straight-forward computations it follows that both (1) and (2) in Definition
\ref{bfspaces1} are fulfilled with $\mascB _{(\omega )}$ in place of $\mascB$ provided $v$
has been modified in suitable ways.
\end{proof}


\par

\begin{proof}[Proof of Proposition \ref{Prop:LargestModSpace}]
Let $\phi \in \Sigma _1(\rr d)\setminus 0$ be fixed, $\mascB _{(\omega )}$
be the Banach space which consists of
all $F\in L^1_{loc}(\rr {2d})$ such that
$$
\nm F{\mascB _{(\omega )}} \equiv \nm {F\cdot \omega}{\mascB}.
$$
Since $\omega$ is a moderate function, it follows by the previous lemma
$\mascB _{(\omega )}$ is an invariant
BF-space.

\par

Let $\{ f_j\} _{j=1}^\infty$ be a Cauchy-sequence in $M(\omega ,\mascB)$.
Then $\{ V_\phi f_j\} _{j=1}^\infty$ is a Cauchy-sequence in $\mascB
_{(\omega )}$. Since $\mascB _{(\omega )}$ is a Banach space, there is
a unique $F\in \mascB _{(\omega )}$ such that
$$
\lim _{j\to \infty}\nm {V_\phi f_j -F}{\mascB _{(\omega )}} =0.
$$
Let $f=\nm \phi{L^2}^{-2}V_\phi ^*F$. Then $V_\phi f=P_\phi F$
belongs to $\mascB _{(\omega)}$, in view of Lemma \ref{Lemma:Proj}
(3). Since $P_\phi$ is continuous on
$\mascB _{(\omega )}$ and satisfies the mapping properties given in
Lemma \ref{Lemma:Proj}, we get
\begin{multline*}
\lim _{j\to \infty} \nm {f_j-f}{M(\omega ,\mascB)}
=
\lim _{j\to \infty} \nm {V_\phi (f_j-f)}{\mascB _{(\omega )}}
\\[1ex]
=\lim _{j\to \infty} \nm {P_\phi (V_\phi f_j - F)}{\mascB _{(\omega )}}
\lesssim
\lim _{j\to \infty} \nm {V_\phi f_j - F}{\mascB _{(\omega )}} =0.
\end{multline*}
Hence, $f_j\to f$ in $M(\omega ,\mascB)$, and the completeness 
of $M(\omega ,\mascB)$ follows. Consequently, $M(\omega ,\mascB)$
is a Banach space.

\par

The embedding \eqref{Eq:LargestModSpace} is an immediate consequence
of \cite[Theorem 2.5]{To26} and the fact that $M(\omega ,\mascB)$
is a Banach space.
\end{proof}

\par

If we assume that $\mascB$ is an invariant QBF-space (instead of invariant
BF-space) with respect of $v_0$, then it seems to be an open question 
wether \eqref{Eq:LargestModSpace} might be violated or not.

\par

Before studying compactness of embeddings between
modulation spaces, we first consider the related continuity questions.

\par

\begin{thm}\label{continuous}
Let $\omega _1$ and $\omega _2$ be weights on $\rr {2d}$,
$\mascB$ be an invariant 
BF-space on $\rr {2d}$ with respect to $v \in \mathscr{P}_E$ or a mixed quasi-normed
space of Lebesgue type, and let $i$ be the injection
\begin{equation}\label{embedding2}
i\! : M(\omega _1,\mascB )\to
M(\omega _2,\mascB ).
\end{equation}
Then the following is true:
\begin{enumerate}
\item if ${\omega_2}/{\omega_1}$ is bounded,
then the map \eqref{embedding2} is continuous;

\vrum

\item if in addition $\omega _1,\omega _2\in \mascP
_E(\rr {2d})$ and $v$ is bounded, then the map \eqref{embedding2}
is continuous, if and only if ${\omega_2}/{\omega_1}\leq C$ for some $C>0$.
\end{enumerate}
\end{thm}

\par

The next lemma is related to Remark \ref{Rem:LargestModSpace}
and is needed verify the previous theorem.

\par

\begin{lemma}\label{Lemma:MBInfty}
Let $v$ be submultiplicative and bounded on $\rr {2d}$, $\mascB$
be an invariant BF-space 
with respect $v$ which is continuously embedded
in $\Sigma _1'(\rr {2d})$, and let $\omega
\in \mascP _E(\rr {2d})$. Then
$M(\omega ,\mascB) \hookrightarrow M^\infty_{(\omega)} (\rr d)$.
\end{lemma}

\par
\begin{proof}
Let $\mascB '$ be the $L^2$-dual of $\mascB$. Then it follows by
straight-forward computation that both $\mascB$ and $\mascB '$ are
translation invariant Banach spaces of order $1$ which contain
$\Sigma _1(\rr {2d})$.
%
Let $\phi \in \Sigma _1(\rr d)$ be such that
$\nm \phi {L^2}=1$, and let
$$
\Omega = \sets {g\in \Sigma _1 (\rr d)}{\nm g{M^1_{(1/\omega )}}\le 1}.
$$
Since $(M^1_{(1/\omega )}(\rr d))' = M^\infty _{(\omega )}(\rr d)$ by
a unique extension of the $L^2$-form on $\Sigma _1(\rr d)$ and that
$\Sigma _1$ is dense in $M^1_{(1/\omega )}$, we get
\begin{multline*}
\nm f{M^\infty _{(\omega )}}
\asymp
\sup _{g\in \Omega} |(f,g)_{L^2(\rr d)}|
=
\sup _{g\in \Omega} |(V_\phi f\cdot \omega ,V_\phi g/\omega )_{L^2(\rr {2d})}| 
\\[1ex]
\le
\sup _{g\in \Omega} \nm {V_\phi f\cdot \omega }{\mascB}
\nm {V_\phi g/\omega}{\mascB '}
\lesssim
\nm {V_\phi f\cdot \omega }{\mascB}
\asymp
\nm f{M(\omega ,\mascB )}.
\end{multline*}
Here we have used the fact that
$$
\nm {V_\phi g/\omega}{\mascB '}
\asymp
\nm g{M(1/\omega ,\mascB ' \, )}
\lesssim
\nm g{M^1_{(1/\omega )}}<\infty ,
$$
which follows from
Feichtinger's minimality principle (cf.
the extension \cite[Theorem 2.4]{To26} of
\cite[Theorem 12.1.9]{Gc1}).
%
%
%
%
\end{proof}

The previous lemma enables us to verify Theorem \ref{continuous}:

\begin{proof}[Proof of Theorem \ref{continuous}]

Claim (1) is an immediate consequence of the boundedness of
${\omega_2}/{\omega_1}$ and of $\mathscr{B}$ being an invariant BF-space. 

\medspace

Assume instead that the embedding $i$ in \eqref{embedding2} is
continuous and all assumptions of the second claim hold. Claim (2) follows
if we have proved the boundedness of
$\omega _2/\omega _1$. We prove this boundedness by 
contradiction. We consider, that there is a sequence $(x_k,\xi _k)
\in \rr {2d}$ with $|(x_k,\xi _k)| \rightarrow \infty$ if
$k \rightarrow \infty$ fulfilling
\begin{align*}
    \frac{\omega_2(x_k,\xi _k)}{\omega _1(x_k,\xi _k)}
    \geq k \qquad \text{for all } k \in \mathbf {N}.
\end{align*}
Let $\phi$ be as in \eqref{Eq:BargmannSTFTlink} and set
$$
f_k=\frac{1}{\omega_1(X_k)}
e^{i\langle \cdo ,\xi _k\rangle}\phi (\cdo
-x_k),\qquad X_k=(x_k,\xi _k).
$$

\par

In order to show that the sequence $f_k$ is bounded in
$M{(\omega_1,\mascB)}$, we choose a submultiplicative
weight $v_0\in \mascP_E (\rr {2d})$ such that $\omega _1$
is $v_0$-moderate and that $v_0\geq 1$.

\par

By
$$
V_{\phi}(e^{i \scal \cdo \xi}f(\cdo -x))(y,\eta)
=
e^{i\scal x{\eta-\xi}} (V_{\phi}f)(y-x, \eta-\xi)
$$
(see e.{\,}g. \cite{Gc2}), we get
\begin{align*}
\|e^{i\langle \cdo ,\xi \rangle} f(\cdo
-x)\|_{M(\omega_1, \mascB)}
\leq C \omega_1(x,\xi )
\| f\|_{M(v_0,\mascB)},\qquad f\in M{(v_0, \mascB)}.
\end{align*}
This gives 
\begin{align*}
\|f_k \| _{M(\omega_1, \mascB)}
=\frac{1}{\omega_1(X_k)}\| e^{i\langle \cdo ,\xi _k\rangle}\phi (\cdo
-x_k)\| _{M(\omega_1, \mascB)}
\le C \| \phi \|_{M(v, \mascB)}< \infty,
\end{align*}
where $C$ is independent of $k \in \mathbb{N}$.
Then the hypothesis provides the boundedness of the sequence $\{ f_k\}$
in ${M(\omega_2, \mascB)}$.

\par

Since $M(\omega_2, \mathscr{B}) \hookrightarrow M^\infty_{(\omega_2)}$ due
to Lemma \ref{Lemma:MBInfty} we
have
\begin{equation}\label{eq1}
\sup_{X\in\rr {2d}}\omega_2(X)|(V_\phi (f_k))(X)|\leq
C\|f_k\|_{M(\omega_2, \mascB)}\leq C \qquad \text{for all } k \in \mathbb{N}
\end{equation}
for some $C>0$. In particular  inequality (\ref{eq1}) yields if we take $z=z_k$
\begin{multline}
\omega_2(X_k)|(V_\phi (f_k))(X_k)|
=
\frac{\omega_2(X_k)}{\omega_1(X_k)}|(V_\phi (e^{i\langle \cdo
,\xi _k\rangle}\phi (\cdo -x_k))(X_k)|
\\
=
\frac{\omega_2(X_k)}{\omega_1(X_k)}|(V_\phi \phi )(0)|
=
(2\pi)^{-\frac d2}\, \frac{\omega_2(X_k)}{\omega_1(X_k)}
\leq
C
\end{multline}
which proves the result.
\end{proof}

We have now the following extension of \cite[Theorem 1.2]{BoTo},
which is our main result.

\par

\begin{thm}\label{compact}
Let $\omega _1,\omega _2\in \mascP _Q(\rr {2d})$,
$v \in \mascP _E(\rr {2d})$ be submultiplicative,
$\mascB$ be an invariant 
BF-space on $\rr {2d}$ with respect to $v$ or a mixed quasi-normed
space of Lebesgue type, and let $i$ be the injection
\begin{equation}\label{embedding1}
i\! : M(\omega _1,\mascB )\to
M(\omega _2,\mascB ).
\end{equation}
Then the following is true:
\begin{enumerate}
\item if ${\omega_2}/{\omega_1}\in L^\infty _0(\rr {2d})$,
then the map \eqref{embedding1} is compact;

\vrum

\item if in addition $\omega _1,\omega _2\in \mascP
_E(\rr {2d})$ and $v$ is bounded, then the map \eqref{embedding1}
is compact, if and only if ${\omega_2}/{\omega_1}\in L^\infty
_0(\rr {2d})$.
\end{enumerate}
\end{thm}

\par

We need the following lemma for the proof.

\par
\begin{lemma}\label{Lemma:UniformSeq}
Let $\mascB$ be an invariant BF space on $\rr {2d}$,
$\phi (x)=\pi ^{-\frac d4}e^{-\frac 12\cdot |x|^2}$,
$x\in \rr d$, $\omega \in \mascP _Q(\rr {2d})$ and
let $\{ f_j\} _{j=1}^\infty \subseteq \Sigma _1'(\rr d)$
be a bounded set in $M(\omega ,\mascB )$. Then there is a
subsequence $\{ f_{j_k}\} _{k=1}^\infty$ of $\{ f_j\}
_{j=1}^\infty$ such that $\{ V_\phi f_{j_k}\} _{k=1}^\infty$
is locally uniformly convergent.
\end{lemma}

\par

\begin{proof}
By the link between the Bargmann transform and Gaussian windowed
short-time Fourier transforms, the result follows if we prove the
assertion with $F_j=\mathfrak V_df_j$ in place of $V_\phi f_j$.
For any $R>0$, let $D_R$ be the poly-disc
\begin{align*}
D_R \equiv &\sets{(x,\xi )\in \rr {2d}}{x_j^2+\xi _j^2<R^2,\ j=1,\dots ,d}
\intertext{in $\rr {2d}$ which we identify with}
&\sets{x+i\xi \in \cc {d}}{x_j^2+\xi _j^2<R^2,\ j=1,\dots ,d}
\end{align*}
in $\cc d$. By 
\eqref{Eq:BargmannSTFTlink}
and an application of Cantor's diagonalization
principle the result follows if we prove that for each $R>0$, there
is a
subsequence $\{ f_{j_k}\} _{k=1}^\infty$ of $\{ f_j\}
_{j=1}^\infty$ such that $\{ F_{j_k}\} _{j=1}^\infty$
is uniformly convergent on $D_R$.

\par

By \cite[Theorem 3.2]{To18}, it follows that $\{ f_j\} _{j=1}^\infty$
is a bounded set in $M^{\infty}_{(\omega _0)}(\rr d)$ for some
choice of $\omega _0\in \mascP _Q(\rr {2d})$. Hence, 
$\{ V_\phi f_{j_k}\} _{j=1}^\infty$ and thereby
$\{ F_{j}\} _{j=1}^\infty$ are locally uniformly
bounded on $\rr {2d}$. In particular,
\begin{equation}\label{Eq:CR}
C_R \equiv \sup _{j\ge 1}\nm {F_j}{L^\infty (D_{2R})}
\quad \text{and}\quad
C_{R,\omega _0} \equiv \sup _{j\ge 1}\nm {F_j\omega _0}
{L^\infty (D_{2R})}
\end{equation}
are finite for every weight $\omega _0$ on $\cc d\simeq \rr {2d}$.

\par

By Cauchy's and Taylor's formulae we have
\begin{align}
F_j(z) &= \sum _{\alpha \in \nn d} a_j(\alpha )z^\alpha ,\qquad
z\in D_R,
\label{Eq:TaylorSerieFj}
\intertext{where}
|a_j(\alpha )| &\le C_R(2R)^{-|\alpha |}.
\label{Eq:TaylorCoeffEst}
\end{align}
In particular, if $\{ \beta _l\} _{l=1}^\infty$ be an enumeration
of $\nn d$, then for each $l\ge 1$, $\{ a_j(\beta _l )\}
_{j=1}^\infty$ is a bounded set in $\mathbf C$. Hence,
for a subsequence $I_1=\{ k_{1,1},k_{1,2},\dots \}$
of $\mathbf Z_+=\{ 1,2,\dots \}$, the limit
$$
\lim _{m\to \infty} a_{k_{1,m}}(\beta _1) 
$$
exists. By induction it follows that for some family of
subsequences
$$
I_N=\{ k_{N,1},k_{N,2},\dots \} \subseteq \mathbf Z _+,
$$
which decreases with $N$, the limit
$$
\lim _{m\to \infty} a_{k_{N,m}}(\beta _n) 
$$
exists for every $n\le N$.

\par

By Cantor's diagonal principle, there is a subsequence 
$\{ {j_k}\} _{k=1}^\infty$ of $\mathbf Z_+$ and sequence
$\{ b(\alpha )\} _{\alpha \in \nn d}$ such that
$$
\lim _{k\to \infty} a_{j_k}(\alpha ) = b(\alpha ).
$$
By \eqref{Eq:TaylorCoeffEst} we get
$$
|b(\alpha )| \le C_R(2R)^{-|\alpha |}.
$$
This in turn gives
\begin{equation}\label{Eq:TaylorTermEst}
\sup _{j\ge 1}\nm {a_j(\alpha )z^\alpha}{L^\infty (D_R)}
\le
C_R2^{-|\alpha |}
\quad \text{and}\quad
\nm {b(\alpha )z^\alpha}{L^\infty (D_R)}
\le
C_R2^{-|\alpha |}
\end{equation}
Hence, \eqref{Eq:TaylorSerieFj} and the Taylor series
$$
F(z) \equiv \sum _{\alpha \in \nn d}b(\alpha )z^\alpha ,
$$
are uniformly convergent on $D_R$, and by using
\eqref{Eq:TaylorTermEst}, it follows by
straight-forward computations that $F_{j_k}$ tends to $F$ uniformly
on $D_R$ when $k$ tends to infinity.
\end{proof}

\par

\begin{proof}[Proof of Theorem \ref{compact}]
In order to verify (1) we need to show, that a bounded sequence $\{f_j\}$ in
$M(\omega _1,\mascB )$ has a convergent subsequence in
$M(\omega _2,\mascB )$.
By means of the assumptions there is a sequence of increasing balls $B_k$,
$k\in \mathbf Z_+$,
centered at the origin with radius tending to $+\infty$ as $k\to
\infty$ such that
\begin{equation}\label{n.1}
\frac{\omega_2(x,\xi)}{\omega_1(x,\xi)}\le\frac{1}{k},\quad
\text{when}\quad (x,\xi )\in \rr {2d}\setminus B_k.
\end{equation}

\par

By Lemma \ref{Lemma:UniformSeq}
it follows that if $\phi (x)=\pi ^{-\frac d4}e^{-\frac 12\cdot |x|^2}$,
$x\in \rr d$, then there is a
subsequence $\{ h_j\}_{j=1}^\infty$ of $\{ f_{j}\}_{j=1}^\infty$
such that $\{ V_\phi h_j\}_{j=1}^\infty$ converges uniformly on any $B_k$,
and converges on the whole $\rr {2d}$.

\par

We have to prove that
$\nm {h_{m_1} - h_{m_2}}{M(\omega_2,\mascB )}\to 0$ as
$m_1,m_2\to \infty$. Let 
$\chi _k$ be the characteristic function of $B_k$, $k\ge 1$.
From the fact that $C_R$ in \eqref{Eq:CR} is bounded we have
\begin{multline}\label{estimate}
\begin{aligned}
\| &h_{m_1} - h_{m_2} \| _{M(\omega_2,\mascB )} = \nm
{V_\phi h_{m_1}-V_\phi h_{m_2}}{\mascB _{(\omega _2)}}
\\[1ex]
&\le
\|(V_\phi h_{m_1}-V_\phi h_{m_2})\chi _{k}\|_{\mascB _{(\omega_2)}}+
{\|V_\phi h_{m_1}-V_\phi h_{m_2}\|_{\mascB _{(\omega_1)}}}/{k}
\\[1ex]
&\le
\|(V_\phi h_{m_1}-V_\phi h_{m_2})\chi _{k}\|_{\mascB _{(\omega_2)}}+
{2C}/{k},
\end{aligned}
\end{multline}
where $C=C_R$ is the constant in \eqref{Eq:CR}.

\par

In order to make the right-hand side arbitrarily small, $k$ is first
chosen large enough. Then $V_\phi h_1, V_\phi h_2,\dots$ is a sequence
of bounded continuous functions converging uniformly on the compact set
$\overline B_k$. Since $\omega_2$ is a weight and $\mascB$ is an
invariant BF-space we obtain
\begin{align*}
    &\|\left( V_{\phi}h_{m_1}-V_{\phi}h_{m_2} \right) \chi_{k}
    \|_{\mascB_{(\omega_2)}}
    =\|\left( V_{\phi}h_{m_1}-V_{\phi}h_{m_2} \right)\omega_2 \chi_{k}
    \|_{\mascB}
    \\ 
    &\qquad\qquad \lesssim
    \left (
    \sup_{(x,\xi) \in B_k} |\left(
    V_{\phi}h_{m_1}(x,\xi)-V_{\phi}h_{m_2}(x,\xi) \right)
    \omega_2(x, \xi)|
    \right )
    \|\chi_{k}\|_{\mascB}
    \\
    &\qquad \qquad \lesssim \sup_{(x,\xi) \in B_k}
    |V_{\phi}h_{m_1}(x,\xi)-V_{\phi}h_{m_2}(x,\xi)|
\end{align*}
tends to zero as $m_1$ and $m_2$ tend to infinity. This proves (1).

\medspace

In order to verify (2) we suppose that the embedding $i$ in \eqref{embedding1} is
compact and all assumptions of the second claim hold. From the first part of the
proof, the result follows if we prove that
$\omega _2/\omega _1$ turns to zero at infinity. We prove this claim by 
contradiction.

\par

Suppose there is a sequence $(x_k,\xi _k)
\in \rr {2d}$ with $|(x_k,\xi _k)| \rightarrow \infty$ if
$k \rightarrow \infty$ and a $C>0$ fulfilling
\begin{equation}\label{Eq:ContradIneq}
    \frac{\omega_2(x_k,\xi _k)}{\omega _1(x_k,\xi _k)}
    \geq C \qquad \text{for all } k \in \mathbf {N}.
\end{equation}
Let $\phi$ be as in \eqref{Eq:BargmannSTFTlink} and set
$$
f_k=\frac{1}{\omega_1(X_k)}
e^{i\langle \cdo ,\xi _k\rangle}\phi (\cdo
-x_k),\qquad X_k= (x_k,\xi _k).
$$

\par

By the proof of Theorem \ref{continuous}, it follows that the sequence 
$\{ f_k \} _{k=1}^\infty$ is bounded in $M(\omega_1,\mascB)$,
and by the assumptions $\{ f_k \} _{k=1}^\infty$ is precompact in
${M(\omega_2, \mascB)}$.

\par

Let $\phi \in \Sigma_1 (\rr d)$. Then $V_{\phi}\phi
\in \Sigma_1 (\rr {2d})$ by Remark \ref{Rem:STFT}. From the fact
$\omega _1\gtrsim e^{-r_0|\cdo |}$ for some $r_0>0$
we get
$$
\int \phi(x)\overline{f_k(x)}\, dx =
\frac{1}{\omega_1(x_k,\xi _k)}(V_{f_0}\phi)(x_k,\xi _k)\to 0,
$$
as $k \to \infty$, which implies, that $f_k$ tends to zero in
$\Sigma_1 '(\rd)$. Hence the only possible limit point in
$M(\omega_2, \mascB)$ of $\{ f_k \} _{k=1}^\infty$ is zero.

\par

As $\{ f_k \} _{k=1}^\infty$ is precompact in $M(\omega_2, \mascB)$,
we can then extract a subsequence $\{ f_{k_j}\} _{j=1}^\infty$ which converges to zero
in $M(\omega_2, \mascB)$.

\par

Since $M(\omega_2, \mathscr{B}) \hookrightarrow M^\infty_{(\omega_2)}$
due to Lemma \ref{Lemma:MBInfty} we have
\begin{equation}
\sup_{X\in\rr {2d}}\omega_2(X)|(V_\phi (f_{k_j}))(X)|\le
C\|f_{k_j}\|_{M(\omega_2, \mascB)}\to 0
\end{equation}
as $j\to \infty$. Taking $X=X_{k_j}$ in the previous inequality provides
\begin{multline}
\omega_2(X_{k_j})|(V_\phi (f_k))(X_{k_j})|
=
\frac{\omega_2(X_{k_j})}{\omega_1(X_{k_j})}|(V_\phi (e^{i\langle \cdo
,\xi _{k_j}\rangle}\phi (\cdo -x_{k_j}))(X_{k_j})|
\\
=
\frac{\omega_2(X_{k_j})}{\omega_1(X_{k_j})}|(V_\phi \phi)(0)|
(2\pi )^{-\frac d2}\, \frac{\omega_2(X_{k_j})}{\omega_1(X_{k_j})}
\to 0.
\end{multline}
which contradicts \eqref{Eq:ContradIneq} and proves (2).
\end{proof}

\par

As an immediate consequence of Lemma
\ref{compares1} and Theorem \ref{compact} we get:

\par

\begin{cor} Assume that $\omega_1,\omega_2\in \mascP (\mathbf
R^{2n})$, and that $p,p_0,q,q_0\in[1,\infty]$ such that
$p_0,q_0<\infty$. Assume also that ${\omega_2}/{\omega_1}\in
L^{p_0,q_0} (\rr {2d})$. Then the embedding \eqref{embedding1}
is compact.
\end{cor}

\par

\medspace

\end{document}